\documentclass[reqno,12pt]{amsart}
\usepackage{mathrsfs}
\usepackage{amscd}
\usepackage{amsmath}
\usepackage{latexsym}
\usepackage{amsfonts}
\usepackage{amssymb}
\usepackage{amsthm}
\usepackage{graphicx}
\usepackage{color}

\def\H{{\cal H}}

\def\R{\mathbb{R}}

\def\H2{H^2(\R^N)}
\def\L2{L^2(\R^N)}
\def\to{\rightarrow}
\def\cd{\!\cdot\!}

\numberwithin{equation}{section}

\newtheorem{thm}{Theorem}[section]

\newtheorem{lem}[thm]{Lemma}

\newtheorem{prop}[thm]{Proposition}

\newtheorem{cor}[thm]{Corollary}
\newtheorem{remark}{Remark}[section]

\makeatletter
\newcommand{\Extend}[5]{\ext@arrow0099{\arrowfill@#1#2#3}{#4}{#5}}
\makeatother

\begin{document}

\setcounter{page}{1}

\title[Blow up for NLS]{On Blow-up criterion for the Nonlinear Schr\"{o}dinger Equation}
\author{Dapeng Du}
\address{School of Mathematics and Statistics, Northeast Normal University, \ Changchun, \ P.R.China, \ 130024,} \email{dudp954@nenu.edu.cn}
\thanks{The first author was partially supported by the Chinese NSF (No. 11001043) and the China Postdoctoral Science Foundation (No. 20090460074).}

\author{Yifei Wu}
\address{School of  Mathematical Sciences, Beijing Normal University, Laboratory of Mathematics and Complex Systems,
Ministry of Education, Beijing, P.R.China, 100875}
\email{yerfmath@gmail.com}
\thanks{The second author was partially supported
by the Chinese NSF (No. 11101042) and the Chinese Postdoctoral
Science Foundation (No. 2012T50068).}

\author{Kaijun Zhang}
\address{School of Mathematics and Statistics, Northeast Normal University, \ Changchun, \ P.R.China, \ 130024,} \email{zhangkj201@nenu.edu.cn}

\thanks{The third author was partially supported by the Chinese NSF (No. 11071034) and the Fundamental Research Funds for the Central Universities (No. 111065201).}




\maketitle

\begin{abstract}\noindent
The blowup is studied for the nonlinear Schr\"{o}dinger equation
$iu_{t}+\Delta u+ |u|^{p-1}u=0$ with $p$ is odd and $p\ge 1+\frac 4{N-2}$ (the
energy-critical or energy-supercritical case). It is shown that the
solution with negative energy $E(u_0)<0$ blows up in finite or
infinite time. A new proof is also presented for the previous result
in \cite{HoRo2}, in which  a similar result but more general in a
case of energy-subcritical was shown.
\end{abstract}

\section{Introduction}
The  Schr\"{o}dinger equation is the fundamental equation in quantum
mechanics. Its general form is
\begin{equation}\label{EQS-NLS-General}
    iu_{t}+\Delta u-Vu=0,
\end{equation}
where $V$ denote the potential and $|u|^2$ is the probability density
that the particle appears at the point $(x,t)$. The
solution $u$ is called wave function. In this paper, we study
the following well-known
focusing nonlinear
Schr\"{o}dinger equation
\begin{equation}\label{EQS-NLS}
   \left\{ \aligned
    &iu_{t}+\Delta u+ |u|^{p-1}u=0,\quad\;  (x,t)\in \R^N\times\R,\\
    &u(x,0)=u_0(x), \;\qquad\quad  x\in \R^N.
   \endaligned
  \right.
\end{equation}
This equation received a great deal of attention from
mathematicians, in particular because of its applications to
nonlinear optics, see for examples, Berg\'{e} \cite{Be-phy-98},
Sulem and Sulem \cite{Susu-Phy-99}. For \eqref{EQS-NLS}, the
potential $V=-|u|^{p-1}$. Notice that $V$ depends on the wave
function $u$. This give the term \emph{nonlinear}. The potential $V$
becomes negative very large when the probability density $|u|^2$ is very
large. This property brings another term \emph{focusing}. The
equation \eqref{EQS-NLS} has very important scaling invariant
symmetry:
\begin{equation}\label{eqs:scaling}
u_\lambda(x,t) = \lambda^{\frac2{p-1}} u(\lambda x, \lambda^2 t),
\end{equation}
in the sense that both the equation and the $\dot{H}^{s_c}$-norm are
invariant under the scaling transformation, where
\begin{equation}\label{eqs.1:sc}
   s_c=\frac N2-\frac2{p-1}.
\end{equation}
This gives the notation \emph{critical regularity} , the lowest
regularity assumption that the equation \eqref{EQS-NLS} is
well-posed. If the critical regularity of the problem
\eqref{EQS-NLS} is higher/lower than $s$, we call the problem
$\dot{H}^s$-\emph{subcritical/supercritical}. In view of this, the
Schr\"{o}dinger equation \eqref{EQS-NLS} is called
energy-subcritical when $p<1+\frac{4}{N-2}$, which is equivalent to
$s_c<1$ (in particular, it is called the mass-critical when
$p=1+\frac{4}{N}$ or $s_c=0$); it is called energy-critical when
$p=1+\frac{4}{N-2}$, which is equivalent to $s_c=1$; and it is
called energy-supercritical when $p>1+\frac{4}{N-2}$, which is
equivalent to $s_c>1$.

The solution of equation \eqref{EQS-NLS} obeys the mass, momentum and energy
conservation laws, which read as
\begin{equation}\label{eqs:energy-mass}
   \aligned
M(u(t))&\triangleq\int |u(x,t)|^2\,dx=M(u_0),\\
P(u(t))&\triangleq\textrm{Im}\int \overline{u(x,t)}\nabla u(x,t)\,dx=P(u_0),\\
E(u(t)) &\triangleq \int |\nabla u(x,t)|^2 - \frac2{p+1}\int
|u(x,t)|^{p+1} \,dx = E(u_0).
   \endaligned
\end{equation}

The local well-posedness for the initial data problem
\eqref{EQS-NLS} with $u_0\in H^1(\R^N)$ was studied in Cazenave and
Weissler \cite{CaWe-NA-93} in the energy-subcritical/critical cases.
It was also shown in \cite{KiVi10} that the problem \eqref{EQS-NLS}
in the energy-supercritical case is locally well-posed under some
assumptions on the dimension $N$ and the power $p$.
A natural question is whether the local solution exists globally. In
the mass-subcritical case, it follows easily from the
Gagliardo-Nirenberg inequality that the global solution exists.
From the global theory for small
data, we know that if the Sobolev norm ($H^{s_c}$-norm) of the
initial data is sufficiently small, then there exists a unique
global solution to (1.1). However, for large
 initial data, under suitable smoothness and
decay assumptions, the virial identity guarantees that finite
time blowup may occur. In particular, Glassey \cite{Glassey77}
proved that if initial data satisfies $xu_0\in L^2(\R^N)$ with
negative energy, then the corresponding solution blows up in finite
time.

After this result, many attempts  have been made to remove/relax the
finite variance assumption. Especially, in the energy-subcritical
case, Ogawa and Tsutsumi \cite{OgTs91} removed the finite variance
assumption in the radial symmetry case. The radiality condition was
relaxed to some nonisotropic ones by Martel \cite{Ma97}. In the 1D
mass-critical case ($p=5$), Ogawa and Tsutsumi \cite{OgTs92}
completely removed the finite variance assumption. As a remark in
the famous paper \cite{MeRa}, Merle and Raphael showed that in the
mass-critical case, if the mass of the initial data is close to the
mass of the ground state, then the solution with negative energy
blows up in finite time. The similar result was obtained by Raphael
and Szeftel \cite{RaSz09} for the radial quintic nonlinear
Schr\"{o}dinger equation in any dimension:
$$
iu_t+\Delta u+|u|^4u=0, (x,t)\in \R^N\times \R.
$$
%

Besides the finite time blow-up criterion, the other interesting
topic is to see what happens if one only assumes that the initial data
has negative energy. In \cite{GlMe-95}, Glangetas and Merle proved that in the
mass-critical/mass-supercritical, energy subcritical cases with $E(u_0)<0$,  the solution blows up in
finite or infinite time, in the sense of
\begin{equation*}
\sup\limits_{t\in (-T_-(u_0),T_+(u_0))}\|u(t)\|_{H^1}=+\infty,
\end{equation*}
where $(-T_-(u_0),T_+(u_0))$ is the maximal lifespan of the solution
with the initial data $u_0$. The method is a geometrical approach.
See also Nawa \cite{Na99} in the
mass-critical case.
In particular, when $N=3, p=3$, a similar but more general result was
established by Holmer and Roudenko \cite{HoRo2} using the
concentration-compactness argument, see also \cite{CaGu-CM-11, Guoqing} for some related results by using the argument in \cite{HoRo2}. However, it's not clear how to
generalize the argument to the energy-critical/energy-supercritical
cases. In this paper, we give a new argument to prove it. Our
argument is suitable for the energy-critical/energy-supercritical
cases, and gives a similar result about it.

Here comes our theorem, which is about energy-critcal/energy-supercritical cases. For the sake
of simplicity, we only focus our attention on the odd values of the
power $p$.
\begin{thm}\label{main-thm} Suppose that $p$ is odd, $p\ge 1+4/(N-2)$, $N\ge 3$, and
$s>s_c$. Let the initial data $u_0\in H^s(\R^N)$ with $E(u_0)<0$,
and let $u$ be the corresponding solution with the lifetime
$[0,T_{max})$. Then one of the following two statements holds true,
\begin{itemize}
  \item $T_{max}<\infty$, that is, the solution blows up in finite
  time.
  Moreover,
  $$
  \lim\limits_{t\uparrow T_{max}}\|u(t)\|_{H^s}=\infty.
  $$
  \item $T_{max}=\infty$, and there exists a time sequence $\{t_n\}$ such
  that $t_n\to \infty$, and for any $ q>p+1$,
  $$
  \lim\limits_{t_n\uparrow \infty}\|u(t_n)\|_{L^q}=\infty.
  $$
\end{itemize}
A similar result remains true for negative time.
\end{thm}

\begin{remark}
Roughly speaking, Case 1 refers to the finite time blow-up, Case 2
refers to the infinite time blow-up (one may certainly substitute
$L^q$-norm to $H^s$-norm in this case, by Sobolev's embedding). At
this stage, it is not clear whether Case 2 could be ruled out, or it
would indeed happen.
\end{remark}


Thanks to the Galilei transformation, we may extend the negative energy condition to the following.
\begin{cor}\label{cor:main-thm}
Theorem \ref{main-thm} still holds true when the condition
$E(u_0)<0$ is reduced to
\begin{equation}\label{eqs1.1:Q-condition'}
E(u_0)<P(u_0)^2\big/M(u_0).
\end{equation}
\end{cor}

Besides the energy-critical and energy-supercritical cases, our
method also could be used in the energy-subcritical case, that is,
$p<1+4/(N-2)$. Let $Q$ be the ground state of the nonlinear elliptic
equation
\begin{equation}\label{eq:Q}
-Q + \Delta Q + |Q|^{p-1}Q=0, \qquad Q=Q(x), \qquad x\in \R^N.
\end{equation}
As mentioned above, Holmer, Roudenko \cite{HoRo2} and Guo \cite{Guoqing}  proved the
following result.
%
\begin{thm}\label{thm:main_subciritcal}
Let $1+\frac 4N<p< 1+\frac 4{N-2}$, $u$ be the solution of \eqref{EQS-NLS} with the lifetime $[0,T_{max})$, and let the initial data $u_0\in H^1(\R^N)$. Then if
\begin{equation}\label{sub-condition}
   \aligned
M(u_0)^{1-s_c}E(u_0)^{s_c}<M(Q)^{1-s_c}E(Q)^{s_c}, \quad
\|u_0\|_{L^2}^{1-s_c}\|\nabla
u_0\|_{L^2}^{s_c}>\|Q\|_{L^2}^{1-s_c}\|\nabla Q\|_{L^2}^{s_c},
   \endaligned
\end{equation}
then  one of the following
two statements holds true,
\begin{itemize}
  \item $T_{max}<\infty$, and
  $$
  \lim\limits_{t\uparrow T_{max}}\|\nabla u(t)\|_{L^2}=\infty.
  $$
  \item $T_{max}=\infty$, and there exists a time sequence $\{t_n\}$ such
  that $t_n\to \infty$, and
  $$
  \lim\limits_{t_n\uparrow \infty}\|\nabla u(t_n)\|_{L^2}=\infty.
  $$
\end{itemize}
\end{thm}
\begin{remark}
Using energy conservation, it's easy to see that in Theorem
\ref{thm:main_subciritcal} the blow-up norm $\|\nabla u(t)\|_{L^2}$
could be improved to $\|u(t)\|_{L^q}$ for any $q\ge p+1$. But this
is not the case in the energy-supercritical.
\end{remark}
In this paper, we  give a simplified proof, which will be presented
in Section 3.

To prove the main theorems, we  adopt the idea of Glassey
\cite{Glassey77}. Because in our case, the initial data may not have
finite variance, we shall deal with localized virial identities. There are some technical difficulties,
which could be overcome by one observation and two techniques borrowed from scattering theory.
The observation is that the gradient part in
the localized virial identities could be controlled by the gradient
part in the energy.
The first technique is the small $L^2$-estimate in
the exterior ball. It holds true in the relatively long time, which
depends on the radius of the ball. The second is the following
elementary estimate. Suppose $f\in L^1$, then
$$
\int_{|x|<R}|x|^k|f|\,dx =o(R^k),  \textrm{ as } R\to\infty.
$$
Note that one may not expect that the small $L^2$-estimate  in the
exterior ball keeps being right all the time. However, the time
period, in which the small $L^2$-estimate holds true, is long enough
to complete the proof.

This paper is organized as follows. In section 2, we
give the proof of Theorem \ref{main-thm} and Corollary
\ref{cor:main-thm}. Finally we prove Theorem
\ref{thm:main_subciritcal} in Section 3.

\section{Proof of Theorem \ref{main-thm}}

The major part of this section is the following theorem: Theorem
\ref{thm:2.1}, one corollary of which is Theorem \ref{main-thm}.
Before stating this theorem, we introduce some quantities. Let the
quantity
\begin{equation}\label{eqs1.1:Q-def}
    Q(u)\triangleq\int |\nabla u(x)|^2\,dx - \frac {N(p-1)}{2(p+1)}\int
    |u(x)|^{p+1}\,dx,
\end{equation}
then it is well-known as the virial identity that for the solution
$u$ of the equation \eqref{EQS-NLS},
$$
\frac {d^2}{dt^2}\int |x|^2|u(t,x)|^2\,dx=8Q(u(t)).
$$
It implies by Glassey's argument (see \cite{Glassey77}) that the
solution $u$ blows up in finite time if $xu_0\in L^2(\R^d)$ and
there exists $\beta_0<0$ such that
\begin{equation}\label{eqs1.1:Q-condition}
\sup\limits_{t\in (0,T_{max})}Q(u(t))\le \beta_0<0.
\end{equation}

\begin{thm}\label{thm:2.1}
Let $N, p, s$ be the same as in Theorem \ref{main-thm} or Theorem
\ref{thm:main_subciritcal}. Then if there exists $\beta_0<0$ such
that \eqref{eqs1.1:Q-condition} holds, there exists no global
solution $u\in C(\R^+; H^s)$ with
\begin{equation}\label{eqs2.1:050516:53}
\sup\limits_{t\in \R^+}\|u(t,\cdot)\|_{L^q_x}<\infty, \quad \mbox{
for some } q>p+1.
\end{equation}
\end{thm}

\subsection{The Local Theory}
In this subsection, we establish the following local result on the
problem \eqref{EQS-NLS}.
\begin{prop}[Local existence]\label{prop:local theory} Let $s\ge s_c$, and $N, p, s_c$ be as in Theorem \ref{main-thm}. Then for any $u_0\in H^s(\R^N)$,
there exists a unique local solution $u\in C([0,T);H^s(\R^N))$ of
\eqref{EQS-NLS}. Moreover, if $s>s_c$, the lifetime $T$ is only
dependent on $\|u_0\|_{H^s}$.
\end{prop}
\begin{proof} Since the argument is standard, see c.f.
\cite{CaWe-NA-93}, \cite{KiVi10}, we give the proof much briefly.
More generally, we may consider $p>s$ or $p$ is odd.  Let
$I=[0,\delta]$, for some small $\delta>0$ decided later. According
to the Duhamel formula, for $F(u)=|u|^{p-1}u$, we define
$$
\Phi(u(t))=e^{it\Delta}u_0+\int_0^t
e^{i(t-\tau)\Delta}F(u(\tau))\,d\tau.
$$
Let the Strichartz space
$$
SN_s=\bigcap\limits_{(\rho,\gamma,\sigma)\in \Lambda_s}L^\rho_t
W^{\sigma,\gamma}_x(I\times \R^N), \quad
\Lambda_s=\{(\rho,\gamma,\sigma): \frac 2\rho+\frac
d\gamma-\sigma=\frac d2-s,2\le \rho,\gamma,\le \infty\}.
$$
Making using of Strichartz estimates (see \cite{GiVe-strichartz},
\cite{KeTa-Strichartz}) and Sobolev inequality, we have
$$
\|\Phi(u)\|_{SN_s}\le
\|e^{it\Delta}u_0\|_{SN_s}+\||\nabla|^{s}F(u)\|_{L^{q'_0}_tL^{p'_0}_x(I\times\R^N)},
$$
where $\frac 2q_0+\frac N{p_0}=\frac N2, 2\le q_0\le \infty, 2\le
p_0 <\infty$.
 Then the proposition follows by the
standard fixed point theory (in which for the sake of convenience
one may choose the weaker norm $L^\rho_t L^{\gamma}_x(I\times \R^N)$
, for some $(\rho,\gamma,0)\in \Lambda_s$ to be the distance, in
order to avoiding differentiating), once we establish
\begin{equation}\label{NL-estimate}
    \||\nabla|^{s}F(u)\|_{L^{q'_0}_tL^{p'_0}_x(I\times\R^N)}\le
    C\|u\|_{SN_s}\|u\|_{SN_{s_c}}^{p-1}.
\end{equation}
Indeed, it easily follows from the chain rule and H\"{o}lder's
inequality for the regular case, thus we only consider the case when
$0<p-[s]<1$, where we also need additional tool of the fractional
chain rule (see \cite[Lemma 2.6]{KiVi10} for example). In this case,
\begin{equation}\label{eqs:fractional_chain}
\aligned
    \||\nabla|^{s}F(u)\|_{L^{q'_0}_tL^{p'_0}_x(I\times\R^N)}\le C
    \left\||\nabla|^{s-[s]}\big(\nabla^{[s]}F(u)\big)\right\|_{L^{q'_0}_tL^{p'_0}_x(I\times\R^N)}\\
    \le
    C\left\||\nabla|^{s-[s]}\big(F_1(u)\>F_2(u)\big)\right\|_{L^{q'_0}_tL^{p'_0}_x(I\times\R^N)},
\endaligned
\end{equation}
where $F_1(u)$ is a combination of the terms typing as
$$
\partial^{\alpha_1}u\cdots \partial^{\alpha_J}u\cd \partial^{\beta_1}\bar{u}\cdots
\partial^{\beta_K}\bar{u},
$$
for $\alpha_1+\cdots+\alpha_J+\beta_1+\cdots+\beta_K=[s],
  0\le|\alpha_j|,
|\beta_k|\le [s]$ for $ 1\le j\le J, 1\le k\le K;$ and $F_2(u)$ is a
H\"{o}lder continuous function of order $p-[s]$. Then by
H\"{o}lder's inequality and the fractional chain rule,
\begin{equation*}
\aligned
    \eqref{eqs:fractional_chain}\le &C\big\||\nabla|^{s-[s]}F_1(u)\big\|_{L^{q_1}_tL^{r_1}_x}\|F_2(u)\|_{L^{q_2}_tL^{r_2}_x}
    +C\|F_1(u)\|_{L^{q_3}_tL^{r_3}_x}\big\||\nabla|^{s-[s]}F_2(u)\big\|_{L^{q_4}_tL^{r_4}_x}\\
    \le & C \|u\|_{SN_s}\|u\|_{SN_{s_c}}^{[s]-1}\cd
    \|u\|_{SN_{s_c}}^{p-[s]}+
    C \|u\|_{SN_{s_c}}^{[s]}\cd
    \||\nabla|^{\alpha}u\|_{_{L^{q_5}_tL^{r_5}_x}}^{\frac{s-[s]}{\alpha}}\|u\|_{_{L^{q_6}_tL^{r_6}_x}}^{p-[s]-\frac{s-[s]}{\alpha}}\\
    \le &C\|u\|_{SN_s}\|u\|_{SN_{s_c}}^{p-1},
\endaligned
\end{equation*}
where
$$
\frac 1{q'_0}=\frac 1{q_1}+\frac 1{q_2}=\frac 1{q_3}+\frac
1{q_4},\quad\frac 1{p'_0}=\frac 1{r_1}+\frac 1{r_2}=\frac
1{r_3}+\frac 1{r_4}, \quad \big((p-[s])q_2,(p-[s])r_2,0\big)\in
\Lambda_{s_c};
$$
and
$$
\Big([s]q_3, \frac{N-[s]r_3}{N[s]r_3}\Big)\in \Lambda_{s_c},\quad
\frac {s-[s]}{p-[s]}<\alpha<1, \quad (q_5,r_5,0)\Lambda_{s-\alpha},
\quad (q_6,r_6,0)\in \Lambda_{s_c}.
$$
This proves \eqref{NL-estimate} and thus finishes the proof of the
proposition.
\end{proof}

\subsection{The proof of Theorem 2.1}
Roughly speaking, Theorem 2.1 says that there exist no global
solutions whose $L^q$ norms are uniformly bounded in time. We prove
the Theorem 2.1 by contradiction argument. Assume the contrary, then
we have
$$
C_0\triangleq\sup\limits_{t\in \R^+}\|u(t)\|_{L^q}<+\infty.
$$
Then we can show that there exists
$0<\overline{C_0}=\overline{C_0}(C_0, M(u_0), E(u_0))<\infty$, such
that
$$
\overline{C_0}=\sup\limits_{t\in \R^+}\|\nabla u(t)\|_{L^2}.
$$
Indeed, it is first bounded for $L^{p+1}$-norm by interpolation
between $L^{q}$ and $L^2$. Then the boundedness of  $\dot H^1$-norm
follows from the energy conservation law.

Consider the local Virial identity and let
\begin{equation}\label{eqs-I}
I(t)= \int \phi(x)|u(t,x)|^2\,dx,
\end{equation}
then by direct computations (see for examples \cite{Glassey77},
\cite{KeMe06}), one has
\begin{lem} For any $\phi\in C^4(\R^N)$,
\begin{eqnarray}
I'(t)&=&2\textrm{Im} \int \nabla \phi \cd \nabla u
\bar{u}\,dx;\label{eq2.0:eq-I'}\\
I''(t)&=&4\textrm{Re} \sum\limits_{j,k}^N\int
\partial_j\partial_k\phi \cd \partial_j u
\partial_k\bar{u}\,dx-2\frac {p-1}{p+1}\int \Delta \phi
|u|^{p+1}\,dx-\int \Delta^2 \phi |u|^{2}\,dx.\label{eq2.0:eq-I''}
\end{eqnarray}
\end{lem}

If $\phi$ is radial, then one may find that
\begin{equation}\label{eqs-I'}
I'(t)=2\textrm{Im} \int \phi'\, \frac{x\cd \nabla u}{r} \bar{u}\,dx,
\end{equation}
\begin{equation}\label{eqs-I''}
\aligned I''(t)=&4\int \frac{\phi'}{r}|\nabla u|^2\,dx+4\int
\Big(\frac{\phi''}{r^2}-\frac{\phi'}{r^3}\Big)|x\cd\nabla
u|^2\,dx\\
&-2\frac {p-1}{p+1}\int \Big(\phi''+(N-1)\frac{\phi'}{r}\Big)
|u|^{p+1}\,dx-\int \Delta^2 \phi |u|^{2}\,dx,
\endaligned
\end{equation}
here and in the sequel, $r$ denotes $|x|$.

\subsubsection{Virial identity-I and $L^2$-estimate in the exterior
ball}

Fix some large constant $R>0$, which will be decided later, and
choose $\phi$ in \eqref{eqs-I} such that
\begin{equation}\label{eq2.1:eqs-phi-1}
\phi=
   \left\{ \aligned
    &0,\quad\;  0\le r \le R/2,\\
    &1,\quad\;  r\ge R,
   \endaligned
  \right.
\end{equation}
and
$$
0\le\phi\le 1, \quad \phi'\le \frac 4R.
$$
Let $\|u_0\|_{L^2}=m_0$, then by \eqref{eqs-I'},
\begin{equation*}
\aligned I(t)&=I(0)+\int_0^t I'(t')\,dt'\\
&\le I(0)+t \|\phi'\|_{L^\infty}\>\| u\|_{L^2}\>\|\nabla
u\|_{L^2}\\
 &\le \int_{|x|\ge R/2}|u_0|^2\,dx+\frac{4m_0\overline{C_0}t}{R}.
\endaligned
\end{equation*}
Observe that
$$
\int_{|x|\ge R/2}|u_0(x)|^2\,dx=o_R(1),
$$
and
$$
\int_{|x|\ge R}|u(t,x)|^2\,dx\le I(t).
$$
To summarize, we obtain that
\begin{lem} Fixing $\eta_0>0$, then for any $t\le
{\eta_0R}/(4m_0\overline{C_0})$, we have
\begin{equation}\label{lem2.1:L2-local}
\int_{|x|\ge R}|u(t,x)|^2\,dx\le \eta_0+ o_R(1).
\end{equation}
\end{lem}
\begin{remark}
Roughly speaking, the lemma above means that the solution has the
almost finite speed of propagation. To the authors' best knowledge, the
property was first discovered by Lin and Strauss \cite{LiSt} for the
defocusing equation, and widely used in the scattering theory since then. The readers may refer \cite{Ca-book-03,
MiZh-book} for detailed introduction.
\end{remark}

\subsubsection{Virial identity-II}

We rewrite $I''(t)$ in \eqref{eqs-I''} as
\begin{equation}\label{eq2.2:eq-I''}
 I''(t)
=8Q(u(t))+R_1+R_2+R_3,
\end{equation}
and
\begin{equation}
\aligned
R_1=&4\int (\frac{\phi'}{r}-2)|\nabla u|^2\,dx+4\int
\Big(\frac{\phi''}{r^2}-\frac{\phi'}{r^3}\Big)|x\cd\nabla
u|^2\,dx,\\
R_2=& -2\frac {p-1}{p+1}\int
\Big(\phi''+(N-1)\frac{\phi'}{r}-2N\Big)
|u|^{p+1}\,dx,\\
R_3=&-\int \Delta^2 \phi |u|^{2}\,dx.
\endaligned
\end{equation}
Roughly speaking, $R_1,R_2$, and $R_3$ are the error terms from the
localization. We choose $\phi$ such that
\begin{equation}\label{eqs-phi-1}
0\le \phi\le r^2,\quad  \phi''\le 2,\quad  \phi^{(4)}\le
\frac{4}{R^2},
\end{equation}
and
\begin{equation}\label{eqs-phi-2}
\phi=
   \left\{ \aligned
    &r^2,\quad\;  0\le r \le R,\\
    &0,\quad\;  r\ge 2R.
   \endaligned
  \right.
\end{equation}
Then we have
\begin{lem}\label{lem2.2:I''}
There exist two constants $\widetilde{C}(s,p,N,m_0,C_0)>0,
\theta_q>0$, such that
$$
I''(t)\le 8Q(u(t))+\widetilde{C}\|u\|_{L^2(|x|>R)}^{\theta_q}.
$$
\end{lem}
\begin{proof}
We first claim that
\begin{equation}\label{eq2.2:es-R1}
R_1\le 0.
\end{equation}
To prove it, we divide the space $\R^N$ into two parts:
$$\Big\{\frac{\phi''}{r^2}-\frac{\phi'}{r^3}\le 0\Big\}\ \mbox{and}\
\Big\{\frac{\phi''}{r^2}-\frac{\phi'}{r^3}> 0\Big\}.
$$
If $\frac{\phi''}{r^2}-\frac{\phi'}{r^3}\le 0$ it is
obviously true since $\phi'\le 2r$. If
$$
\frac{\phi''}{r^2}-\frac{\phi'}{r^3}\ge 0,
$$
then since $\phi''\le 2$,
\begin{equation*}
\aligned R_1\le &4\int (\phi''-2)|\nabla u|^2\,dx\le 0.
\endaligned
\end{equation*}
So we have proved \eqref{eq2.2:es-R1}. Moreover, since
$$
\textrm{supp}(\phi''+(N-1)\frac{\phi'}{r}-2N)\subset [R,\infty),
$$
by interpolation there exists $0<\theta_q\le 1$, such that
\begin{equation}\label{eq2.2:es-R2}
R_2\le
C\|u\|_{L^q(|x|>R)}^{1-\theta_q}\|u\|_{L^2(|x|>R)}^{\theta_q}\le
CC_0^{1-\theta_q}\|u\|_{L^2(|x|>R)}^{\theta_q},
\end{equation}
where $C>0$, is only dependent on $p,s,N$. Furthermore,
\begin{equation}\label{eq2.2:es-R3}
R_3\le CR^{-2}\|u\|_{L^2(|x|>R)}^2.
\end{equation}
Thus, combining (\ref{eq2.2:eq-I''}) with
(\ref{eq2.2:es-R1})--(\ref{eq2.2:es-R3}), one obtains that for
$R>1$,
$$
I''(t)\le 8Q(u(t))+\widetilde{C}\|u\|_{L^2(|x|>R)}^{\theta_q},
$$
where the constant $\widetilde{C}=\widetilde{C}(s,p,N,m_0,C_0)>0$.
The lemma is now proved.
\end{proof}

\subsubsection{The proof of Theorem \ref{thm:2.1}}

\begin{proof}[Proof of Theorem \ref{thm:2.1}]
Applying (\ref{lem2.1:L2-local}) and Lemma \ref{lem2.2:I''}, one
finds that for any $t\le T\triangleq
{\eta_0R}/(4m_0\overline{C_0})$,
$$
I''(t)\le 8Q(u(t))+\widetilde{C}\big(\eta_0^{\theta_q}+o_R(1)\big).
$$
Integrating from 0 to $T$, and using \eqref{eqs1.1:Q-condition}, one
gets
\begin{equation*}
\aligned I(T)\le &
I(0)+I'(0)T+\int_0^T\int_0^t\Big(8Q(u(t'))+\widetilde{C}\eta_0^{\theta_q}+o_R(1)\Big)\,dt'dt\\
\le &
I(0)+I'(0)T+\Big(8\beta_0+\widetilde{C}\eta_0^{\theta_q}+o_R(1)\Big)\cd\frac{1}{2}T^2.
\endaligned
\end{equation*}
Choosing $\eta_0$ such that
$$\widetilde{C}\eta_0^{\theta_q}=-\beta_0,
$$
and taking $R$ large enough, then for $T=
{\eta_0R}/(4m_0\overline{C_0})$ one has
\begin{equation}\label{eq2.2:Es-IT}
\aligned I(T) \le I(0)+I'(0)
{\eta_0R}/(4m_0\overline{C_0})+\alpha_0R^2,
\endaligned
\end{equation}
where the constant
$$
{\alpha_0}=\beta_0\eta_0^2/(4m_0\overline{C_0})^2<0.
$$
We note here that $\alpha_0$ is independent of $R$.  Now we need the
following two claims:
\begin{equation}\label{eq2.2:Es-I0}
\aligned I(0) =o_R(1)R^2,  \quad I'(0) =o_R(1)R.
\endaligned
\end{equation}
Indeed,
\begin{equation*}
\aligned I(0)\le &
\int_{|x|<\sqrt{R}}|x|^2|u_0(x)|^2\,dx+\int_{\sqrt{R}<|x|<2R}|x|^2|u_0(x)|^2\,dx\\
\le &
Rm_0^2+R^2\int_{|x|>\sqrt{R}}|u_0(x)|^2\,dx\\
= & o_R(1)R^2.
\endaligned
\end{equation*}
A similar argument can be used to obtain the second estimate and
thus proves \eqref{eq2.2:Es-I0}.

Together \eqref{eq2.2:Es-IT}  with \eqref{eq2.2:Es-I0}, and choosing
$R$ large enough, one obtains that for $T=
{\eta_0R}/(4m_0\overline{C_0})$,
\begin{equation*}
\aligned I(T) &\le o_R(1)R^2+{\alpha_0}R^2\\
&\le \frac{1}{2}{\alpha_0}R^2.
\endaligned
\end{equation*}
Since ${\alpha_0}<0$, one finally gets
$$
I(T)<0.
$$
But this is a contradiction with the definition, the proof of
Theorem \ref{thm:2.1} is now completed.
\end{proof}

\subsection{The proof of Theorem \ref{main-thm}}
\label{sec:Proof_mainthm}

\begin{proof}[Proof of Theorem \ref{main-thm}]
From the local theory Proposition \ref{prop:local theory}, we could
define the maximal lifespan $T_{max}$. There are two cases,

\noindent (i) $T_{max}<\infty$. This yields
$$
\lim\limits_{t\to T_{max}}\|u(t)\|_{H^s}=\infty.
$$
Otherwise, there exists a sequence $\{t_n\}_n$ such that $t_n\to
T_{max}$, such that
$$
\sup\limits_{t_n}\|u(t_n)\|_{H^s}<\infty.
$$
Using Proposition \ref{prop:local theory} with the initial data of
$t_n$, we get a contradiction with $T_{max}$ for large $n$.

\noindent (ii) $T_{max}=\infty$. We first observe that
$$
Q(u(t))\le E(u_0)<0, \textrm{ for any } t\in \R.
$$
Thus \eqref{eqs1.1:Q-condition} always holds with $\beta_0=E(u_0)$
under the assumption in this theorem.  Now using Theorem
\ref{thm:2.1}, we prove that there exists a time sequence $\{t_n\}$
such
  that $t_n\to \infty$, and for any $ q>p+1$,
  $$
  \lim\limits_{t_n\uparrow \infty}\|u(t_n)\|_{L^q}=\infty.
  $$
  This concludes Theorem \ref{main-thm}.
\end{proof}

At the end of this section, we give the proof of Corollary
\ref{cor:main-thm}.
\begin{proof}[Proof of Corollary \ref{cor:main-thm}]
From the Galilean transformation,
\begin{equation}\label{eqs1.1:Galilei}
\tilde{u}(t,x) = e^{ix\cdot \xi_0}e^{-it|\xi_0|^2} \, u(t,x-2\xi_0t
).
\end{equation}
If $u$ is the solution of \eqref{EQS-NLS}, then so is $\tilde{u}$.
Moreover, taking $\xi_0=-\frac{P(u_0)}{M(u_0)}$, then
$$
E(u_0)-P(u_0)^2\big/M(u_0)=E(\tilde{u}_0).
$$
Therefore, the conclusion follows by considering $\tilde{u}$
instead.
\end{proof}

\section{The proof of Theorem \ref{thm:main_subciritcal}}

To this end, we shall firstly check that \eqref{sub-condition}
implies \eqref{eqs1.1:Q-condition}, that is, there exists some
strictly negative constant $\beta_0$, such that
\begin{equation*}
\sup\limits_{t\in (-T_-(u_0),T_+(u_0))}Q(u(t))\le \beta_0<0,
\end{equation*}
where $(-T_-(u_0),T_+(u_0))$ is the maximal lifespan. This was
essentially obtained in \cite{HoRo2}, however we also give the proof
here for completeness (with a different argument).

First, we claim that the hypothesis \eqref{sub-condition} implies
that for any $t\in (-T_-(u_0),T_+(u_0))$,
\begin{equation}\label{eqs2.3:sub-condition}
   \aligned
\|u(t)\|_{L^2}^{1-s_c}\|\nabla
u(t)\|_{L^2}^{s_c}>\|Q\|_{L^2}^{1-s_c}\|\nabla Q\|_{L^2}^{s_c}.
   \endaligned
\end{equation}
Indeed, suppose not, then by continuity, there exists $\tilde{t}\in
(-T_-(u_0),T_+(u_0))$, such that
\begin{equation}\label{eqs2.3:sub-argue}
   \aligned
\|u(\tilde{t})\|_{L^2}^{1-s_c}\|\nabla
u(\tilde{t})\|_{L^2}^{s_c}=\|Q\|_{L^2}^{1-s_c}\|\nabla
Q\|_{L^2}^{s_c}.
   \endaligned
\end{equation}
Then by \eqref{eqs2.3:sub-argue} and the sharp Gagliardo-Nirenberg
inequality (see \cite{We83}),
\begin{equation}\label{GN-ineq}
\| u \|^{p+1}_{L^{p+1}} \le C_{\text{GN}} \, \| \nabla u
\|_{L^2}^{\frac{N(p-1)}{2}} \, \| u
\|_{L^2}^{2-\frac{(N-2)(p-1)}{2}},
\end{equation}
where
$$
C_{\text{GN}} = {\| Q \|^{p+1}_{L^{p+1}}}\Big/{\| \nabla Q
\|_{L^2}^{\frac{N(p-1)}{2}} \, \| Q
\|_{L^2}^{2-\frac{(N-2)(p-1)}{2}}},
$$
one obtains that
\begin{equation}\label{eqs2.3:argue-condict}
   \aligned
M(Q)^{\frac{1-s_c}{s_c}}&E(Q)>M(u(\tilde{t}))^{\frac{1-s_c}{s_c}}E(u(\tilde{t}))\\
&= \|u(\tilde{t})\|_{L^2}^{{\frac{2(1-s_c)}{s_c}}}\|\nabla
u(\tilde{t})\|_{L^2}^{2}-\frac2{p+1}\|u(\tilde{t})\|_{L^2}^{{\frac{2(1-s_c)}{s_c}}}\|
u(\tilde{t}) \|^{p+1}_{L^{p+1}}\\
&\ge \|u(\tilde{t})\|_{L^2}^{{\frac{2(1-s_c)}{s_c}}}\|\nabla
u(\tilde{t})\|_{L^2}^{2}\\
&\qquad-\frac2{p+1}C_{\text{GN}}\cd\|u(\tilde{t})\|_{L^2}^{{\frac{2(1-s_c)}{s_c}+2-\frac{(N-2)(p-1)}{2}}}\|\nabla
u(\tilde{t})\|_{L^2}^{\frac{N(p-1)}{2}}\\
&= \|u(\tilde{t})\|_{L^2}^{{\frac{2(1-s_c)}{s_c}}}\|\nabla
u(\tilde{t})\|_{L^2}^{2}-\frac2{p+1}C_{\text{GN}}\cd\Big[\|u(\tilde{t})\|_{L^2}^{{\frac{1-s_c}{s_c}}}\|\nabla
u(\tilde{t})\|_{L^2}\Big]^{\frac{N(p-1)}{2}}\\
&= \|Q\|_{L^2}^{{\frac{2(1-s_c)}{s_c}}}\|\nabla
Q\|_{L^2}^{2}-\frac2{p+1}C_{\text{GN}}\Big[\|Q\|_{L^2}^{{\frac{1-s_c}{s_c}}}\|\nabla
Q\|_{L^2}\Big]^{\frac{N(p-1)}{2}}\\
&=M(Q)^{\frac{1-s_c}{s_c}}E(Q).
   \endaligned
\end{equation}
This gives a contradiction and thus proves
\eqref{eqs2.3:sub-condition}.

By the definition \eqref{eqs:energy-mass} and \eqref{eqs1.1:Q-def},
one has
\begin{equation}\label{eqs:E-Q}
    Q(u(t))=\frac {N(p-1)}{4} E(u(t))-\Big( \frac {N(p-1)}{4}-1\Big)\|\nabla
    u(t)\|_{L^2}^2,
\end{equation}
thus, by \eqref{eqs2.3:sub-condition} and \eqref{sub-condition}, one gives that
$$
Q(u(t))<0,  \textrm{ for any }   t\in (-T_-(u_0),T_+(u_0)).
$$
This together with \eqref{GN-ineq}, and noting that
$\frac{N(p-1)}{2}>2$, yields that there exists some small
$\epsilon_0>0$ such that
\begin{equation}\label{eqs2.3:u-lower-bound}
\|\nabla u(t)\|_{L^2}>\epsilon_0.
\end{equation}
Now we further claim that there exists $\delta_0>0$ such that  for
any $t\in (-T_-(u_0),T_+(u_0))$,
\begin{equation}\label{eqs2.3:Q-upper-bound}
Q(u(t))<-\delta_0\|\nabla u(t)\|_{L^2}^2.
\end{equation}
Indeed, suppose not, there exists a time sequence $\{t_n\}\subset
(-T_-(u_0),T_+(u_0))$ such that
\begin{equation*}
-\delta_n\Big(\frac {N(p-1)}{4}-1\Big)\|\nabla
u(t_n)\|_{L^2}^2<Q(u(t_n))<0,
\end{equation*}
where $\delta_n\to 0$ as $n\to \infty$. Then by \eqref{eqs:E-Q}, one
has
$$
E(u(t_n))>(1-\delta_n)\Big(1-\frac{4}{N(p-1)} \Big)\|\nabla
u(t_n)\|_{L^2}^{2}.
$$
Therefore, by \eqref{eqs2.3:sub-condition}, one finds that
\begin{equation*}
   \aligned
M(u(t_n)&)^{{1-s_c}}E(u(t_n))^{s_c}>(1-\delta_n)^{s_c}M(Q)^{1-s_c}E(Q)^{s_c}.
   \endaligned
\end{equation*}
Thus taking $n\to \infty$ and making use of the mass and energy
conservation laws, we prove that
\begin{equation*}
M(u_0)^{1-s_c}E(u_0)^{s_c}\ge M(Q)^{1-s_c}E(Q)^{s_c}.
\end{equation*}
But this is contradicted with the hypothesis \eqref{sub-condition}
and thus proves \eqref{eqs2.3:Q-upper-bound}. Combining with
\eqref{eqs2.3:u-lower-bound}, we obtain \eqref{eqs1.1:Q-condition}.

Now by Theorem \ref{thm:2.1}, there exists no global solution $u\in
C(\R^+; H^1)$ with \eqref{eqs2.1:050516:53}. Then using Sobolev's
embedding, one may replace $L^q$-norm by $H^1$-norm in
\eqref{eqs2.1:050516:53}, and thus proves Theorem
\ref{thm:main_subciritcal}.

 \hfill$\Box$

\subsection*{Acknowledgements}
The authors would like to express their gratitude to Professor
Changxing Miao for many valuable discussions. The authors also would like to thank Professor
Frank Merle for pointing out the reference \cite{GlMe-95}.

\end{document}